\documentclass[11pt]{amsart}

\usepackage{amsfonts}
\usepackage{amssymb}
\usepackage{hyperref}

\theoremstyle{plain}
\newtheorem{thm}{Theorem}[section]

\newtheorem{prop}[thm]{Proposition}

\theoremstyle{definition}
\newtheorem{defi}[thm]{Definition}
      
\theoremstyle{remark}
\newtheorem{remark}[thm]{Remark}

\newcommand{\lemref}[1]{\hyperref[#1]{Lemma \ref*{#1}}}
\newcommand{\thmref}[1]{\hyperref[#1]{Theorem \ref*{#1}}}
\newcommand{\propref}[1]{\hyperref[#1]{Proposition \ref*{#1}}}
\newcommand{\corref}[1]{\hyperref[#1]{Corollary \ref*{#1}}}
\newcommand{\defref}[1]{\hyperref[#1]{Definition \ref*{#1}}}
\newcommand{\remref}[1]{\hyperref[#1]{Remark \ref*{#1}}}
\newcommand{\conjref}[1]{\hyperref[#1]{Conjecture \ref*{#1}}}

\newcommand{\Inv}{\mathrm{Inv}}
\newcommand{\Aut}{\mathrm{Aut}}

\makeatletter
\newcommand*{\defeq}{\mathrel{\rlap{%
                     \raisebox{0.27ex}{$\m@th\cdot$}}%
                     \raisebox{-0.27ex}{$\m@th\cdot$}}%
                     =}

\numberwithin{equation}{section}
\makeatother

\makeatletter
\def\@setcopyright{}
\def\serieslogo@{}
\makeatother

\begin{document}

\title{An axiomatizable profinite group with infinitely many open subgroups of index 2}

\author{Or Ben Porath}
\address{Raymond and Beverly Sackler School of Mathematical Sciences, Tel-Aviv University, Tel-Aviv, Israel}
\email{orbenporath@mail.tau.ac.il}

\author{Mark Shusterman}
\address{Raymond and Beverly Sackler School of Mathematical Sciences, Tel-Aviv University, Tel-Aviv, Israel}
\email{markshus@mail.tau.ac.il} 
\date{}

\begin{abstract}
We show that a profinite group with the same first-order theory as the direct product over all odd primes $p$ of the dihedral group of order $2p$, is necessarily isomorphic to this direct product.
\end{abstract}

\maketitle

\section{Introduction}

We say that a profinite group $\Gamma$ is axiomatizable if for every profinite group $\Lambda$ with the same first-order theory as that of $\Gamma$, we have $\Lambda \cong \Gamma$.
The study of axiomatizable profinite groups began in \cite{JL}, 
where it is shown that finitely generated profinite groups are axiomatizable, 
and an example of a profinite group that is not axiomatizable is given (for instance, $\mathbb{Z}_2^{\aleph_0}$).
More generally, 
it is shown in \cite{H} that a strongly complete profinite group (that is, a profinite group all of whose finite index subgroups are open) is axiomatizable. 
Are there more axiomatizable profinite groups?

By \cite[Corollary 3.8]{H} a strongly complete profinite group is small (that is, it has only finitely many open subgroups of index $n$, for every $n \in \mathbb{N}$).
Thus, a possible precise formulation given by \cite[Question 3.15 (i)]{H} for the question we have just raised is whether every axiomatizable profinite group is small.
Here, a negative answer to this question is given.

\begin{thm} \label{Res}

The profinite group $G$ given by the direct product of the dihedral groups $D_p$ over all odd primes $p$, is axiomatizable. 

\end{thm}

From our proof of \thmref{Res} one can extract an explicit (infinite) set of axioms characterizing $G$ up to an isomorphism.

\section{The group $G$ and its properties}

\begin{defi} \label{AutDef}
For a profinite group $\Gamma$ we set
\begin{equation}
\Inv(\Gamma) \defeq \{\tau \in \mathrm{Aut}(\Gamma) \ | \ \tau^2 = \mathrm{Id}_ {\Gamma} \}
\end{equation} 
and note that this defines a group if $\Aut(\Gamma)$ is abelian.
\end{defi}

\begin{defi} \label{DihedDef}
For a prime number $p$ we denote by $C_{p}$ the cyclic group of order $p$. 
If $p$ is odd, then $\mathrm{Aut}(C_p) \cong C_{p-1}$ so $\Inv(C_p)$ is a group isomorphic to $C_2$.
The semidirect product $C_p \rtimes \Inv(C_p)$ is the dihedral group $D_p$.
Let $\rho_p$ be a generator of $C_p$, and let $\epsilon_p$ be a generator of $\Inv(C_p)$ so that they generate $D_p$ and we have $\epsilon_p\rho_p\epsilon_p = \rho_p^{-1}$.
\end{defi}

\begin{prop} \label{ComDiProp}
For an odd prime $p$ we have $C_p = \{[a,b] \ | \ a,b \in D_p\}$.
\end{prop}

\begin{proof}
For one inclusion note that $D_p/C_p$ is abelian, and for the other one take some $\rho_p^n \in C_p$.
As $p$ is odd, there exists a $k \in \mathbb{Z}$ such that $2k \equiv n \ (p)$.
We find that $[\epsilon_p, \epsilon_p\rho_p^k] = \epsilon_p\epsilon_p\rho_p^k\epsilon_p\rho_p^{-k}\epsilon_p = \rho_p^k\epsilon_p\rho_p^{-k}\epsilon_p = \rho_p^k\rho_p^k = \rho_p^{2k} = \rho_p^n.$
\end{proof}

\begin{remark} \label{DiRem}
A similar argument shows that $C_p \leq D_p$ is its own centralizer.
\end{remark}

\begin{defi} \label{FundDef}

We set $G \defeq \prod D_p, \ C \defeq \prod C_p, \ E \defeq \prod \langle \epsilon_p \rangle$ where the products (here and in the sequel) are always taken over all odd primes $p$.

\end{defi}

Since $C_p \lhd D_p$ and $D_p/C_p \cong C_2$ for all odd primes $p$, 
we conclude that $C$ is a closed normal procyclic subgroup of $G$ with $G/C \cong (\mathbb{Z}/2\mathbb{Z})^{\aleph_0}$.
Hence,
\begin{equation} \label{SquareEq}
\forall g \in G \ \ g^2 \in C
\end{equation}
and $G/C$ is not small.
Therefore, $G$ is not small as well.
Furthermore, we have $G = C \rtimes E$.
Since the Sylow subgroups of $C$ are normal, we see that
\begin{equation}
\Aut(C) = \prod_p \Aut(C_p) \cong \prod_p C_{p-1}
\end{equation}
is an abelian group, so
\begin{equation}
\Inv(C) = \prod_p \Inv(C_p) = E.
\end{equation}
Thus, 
\begin{equation} \label{GStructEq}
G \cong C \rtimes \Inv(C).
\end{equation}

\begin{defi} \label{ProfComDef}
For a profinite group $\Gamma$ we denote by $\Gamma'$ its profinite commutator, which is the closed subgroup of $\Gamma$
generated by $\{[a,b] \ | \ a,b \in \Gamma \}$.
\end{defi}

It follows from \propref{ComDiProp} that
\begin{equation} \label{CDescEq}
G' = C = \{[a,b] \ | \ a,b \in G \}
\end{equation} 
so the following first-order sentence holds in $G$
\begin{equation} \label{ComLenEq}
\forall a, b, c, d \ \exists r, s \ \ [a,b][c,d] = [r,s]. 
\end{equation}

\begin{remark} \label{SquareRmk}

By \eqref{SquareEq}, \eqref{CDescEq} the following first-order sentence is valid in $G$
\begin{equation} \label{SqEq}
\forall g \ \exists h, k \ \ g^2 = [h,k].
\end{equation}

\end{remark}

\begin{remark} \label{InjRmk}

It follows from \remref{DiRem} that $C$ is its own centralizer in $G$.
In view of \eqref{CDescEq}, this is tantamount to the following first-order sentence
\begin{equation} \label{InjEq}
\forall x \ \Bigg( \Big( \forall y, z \ \ x[y,z]x^{-1}=[y,z] \Big) \longrightarrow \exists a,b \ \ x=[a,b] \Bigg).
\end{equation}

\end{remark}

\begin{remark} \label{SurjRmk}

Fix an odd prime $p$.
We can think of $\epsilon_p, \rho_p$ as elements of $E,C$ respectively, 
and thus also as elements of $G$. 
The first-order sentence
\begin{equation} \label{SurjEq}
\exists x \ \forall y, z \ \ x[y,z]x^{-1} = [y,z]^{-1} \longleftrightarrow [y,z]^{p} = 1 
\end{equation}
holds in $G$ since we can take $x = \epsilon_p$.

\end{remark}

\begin{defi} \label{EEDef}
We say that profinite groups $\Gamma, \Lambda$ are elementarily equivalent if they have the same first-order theory, 
and denote this by $\Gamma \equiv \Lambda$.  
\end{defi}

\section{The proof of \thmref{Res}}
Let $\widetilde{G}$ be a profinite group for which $G \equiv \widetilde{G}$.
Set
\begin{equation} \label{DefCTEq}
\widetilde{C} \defeq \{[g,h] \ | \ g,h \in \widetilde{G}\}
\end{equation} 
and note that it is the image in $\widetilde{G}$ of the compact space $\widetilde{G}^2$ under the continuous map sending $(g,h) \in \widetilde{G}^2$ to $[g,h] \in \widetilde{G}$.
It follows that $\widetilde{C}$ is compact, and thus closed in $\widetilde{G}$.
By \eqref{ComLenEq} $\widetilde{C}$ is a subgroup of $\widetilde{G}$.

Let us now show that $C \equiv \widetilde{C}$.
For that take a first-order sentence $\varphi$ that holds in $C$.
For every variable $x$ that appears in $\varphi$, replace each appearance of $Qx$ by $Q x_1, x_2$ where $Q \in \{\forall, \exists\}$ and $x_1,x_2$ are new variables.
Furthermore, replace any instance of $x$ in any atomic formula in $\varphi$ by $[x_1, x_2]$ and denote the resulting first-order sentence by $\psi$.
It follows from \eqref{CDescEq} that $\psi$ holds in $G$, and thus also in $\widetilde{G}$. 
By \eqref{DefCTEq}, $\varphi$ holds in $\widetilde{C}$ as required.

By \cite[Theorem A]{JL}, $C \cong \widetilde{C}$ and by \eqref{SqEq} every element of $\widetilde{G}/\widetilde{C}$ is of order dividing $2$, so $|\widetilde{G}/\widetilde{C}|$ is prime to $|\widetilde{C}| = |C|$.
Since $\widetilde{C} \cong C$ is abelian and normal in $\widetilde{G}$, 
the action by conjugation of $\widetilde{G}$ on $\widetilde{C}$ gives rise to a continuous homomorphism $\tau \colon \widetilde{G}/\widetilde{C} \to \Inv(\widetilde{C})$.
By the Schur-Zassenhaus theorem (see \cite[Lemma 22.10.1]{FJ}) we get that
\begin{equation} \label{FactorGtilEq}
\widetilde{G} \cong \widetilde{C} \rtimes_\tau \widetilde{G}/\widetilde{C}.
\end{equation}

By \remref{InjRmk} $\tau$ is injective, 
and in order to see that it is also surjective, first identify $\Inv(\widetilde{C})$ with $\Inv(C)$.
A generating set for $\Inv(\widetilde{C})$ is thus given by $\{\epsilon_p\}_p$.
By \remref{SurjRmk} the image of $\tau$ contains $\epsilon_p$ for each odd prime $p$, 
so $\tau$ is a surjection, and thus an isomorphism. 
We conclude that
\begin{equation}
G \stackrel{\ref{GStructEq}}{\cong} C \rtimes \Inv(C) \cong \widetilde{C} \rtimes \Inv(\widetilde{C}) \cong \widetilde{C} \rtimes_\tau \widetilde{G}/\widetilde{C} \stackrel{\ref{FactorGtilEq}}{\cong} \widetilde{G}.
\end{equation}

\section*{Acknowledgments}
We would like to sincerely thank Arno Fehm for telling us about the question that motivated this work, and for many helpful discussions.

\end{document}